\documentclass[11pt]{article}
\usepackage{amssymb}
\usepackage{color}
\usepackage{appendix}
\usepackage{soul}
\newcommand{\bC}{\mathbf{C}}
\newcommand{\bN}{\mathbf{N}}

\newcommand{\bK}{\mathbf{K}}
\newcommand{\ord}{\mbox{\rm ord }}

\newcommand{\ch}{\mathrm{char}\,}

\newlength{\szer}

\newtheorem{defi}{Definition}[section]
\newtheorem{nota}[defi]{Remark}
\newtheorem{ejemplo}[defi]{Example}

\newtheorem{teorema}[defi]{Theorem}

\newtheorem{lema}[defi]{Lemma}
\newtheorem{coro}[defi]{Corollary}

\newenvironment{proof}[1][Proof]{\textbf{#1.} }{\
\rule{0.5em}{0.5em}}

\begin{document}
\title{The Milnor number of plane irreducible singularities in positive characteristic
\footnotetext{
     \noindent   \begin{minipage}[t]{4in}
       {\small
       2010 {\it Mathematics Subject Classification:\/} Primary 32S05;
       Secondary 14H20.\\
       Key words and phrases: plane singularity, Milnor number, degree of conductor, factorization of polar curve.\\
       The first-named author was partially supported by the Spanish Project
       MTM 2012-36917-C03-01.}
       \end{minipage}}}

\author{Evelia R.\ Garc\'{\i}a Barroso and Arkadiusz P\l oski}

\maketitle

\begin{abstract}
\noindent Let $\mu(f)$ resp. $c(f)$ be the Milnor number resp. the degree of the conductor of an
irreducible power series $f\in \bK[[x,y]]$, where $\bK$ is an algebraically closed field of characteristic $p\geq 0$. It is
well-known that $\mu(f)\geq c(f)$. We give necessary and sufficient conditions for the equality $\mu(f)=c(f)$ in terms
of the semigroup associated with $f$, provided that $p>\ord f$.
\end{abstract}

\section*{Introduction}
\label{intro}
\noindent  Let $\bK$ be an algebraically closed field of characteristic $p\geq 0$ and let $f\in \bK[[x,y]]$ be a reduced (without multiple factors)  power series. Denote by $\overline{{\cal O}}$ the normalization of the ring ${\cal O}=\bK[[x,y]]/(f)$ and consider the conductor ideal ${\cal C}$ of $\overline{{\cal O}}$ in ${\cal O}$. The integer $c(f)=\dim_{\bK} {\cal O}/{\cal C}$ is called the degree of the conductor. Since ${\cal O}$ is  Gorenstein we have $c(f)=2\delta(f)$, where $\delta(f)=\dim_{\bK} \overline{{\cal O}}/{\cal O}$ is the double point number. Recall that $\mu(f)=\dim _{\bK} \bK[[x,y]]/\left(\frac{\partial f}{\partial x}, \frac{\partial f}{\partial y}\right)$ is the Milnor number of $f$.

\medskip

\noindent If $\ch \bK=0$ then the Milnor formula holds : $\mu(f)=2\delta(f)-r(f)+1$, where $r(f)$ is the number of distinct irreducible factors of $f$ (see \cite{Milnor}, \cite{Risler}). If the  characteristic $\ch \bK$ is arbitrary then $\mu(f)\geq 2\delta(f)-r(f)+1$ (see \cite{Deligne}, \cite{Melle}) and the equality $\mu(f)=2\delta(f)-r(f)+1$ ($\mu(f)=c(f)$ if $f$ is irreducible) means that $f$ has not {\em wild vanishing cycles}. It is the case if $f$ is Newton non degenerate (see \cite{BGreuel}) or if $p$ is greater than the intersection number of $f$ with its generic polar (see \cite{Nguyen}).

\medskip

\noindent The aim of this note is to give necessary and sufficient conditions for the equality $\mu(f)=c(f)$ in terms of the semigroup associated with the irreducible series $f$, provided that $p> \ord f$ (the order of $f$). Our result gives a partial 
answer to the question raised by G.M. Greuel and Nguyen Hong Duc in \cite{Greuel}.

\section{Main result}
\label{MR}

\noindent Let $f$ be an irreducible power series in $\bK[[x,y]]$, where $\bK$ is an algebraically closed field of characteristic
$p\geq 0$. The semigroup $\Gamma(f)$ associated with the branch $f=0$ is defined as the set of intersection numbers
$i_0(f,h)=\dim_{\bK} \bK[[x,y]]/(f,h)$, where $h$ runs over all power series such that $h \not\equiv 0$ (mod $f$).

\medskip

\noindent Let $\overline{\beta_0},\ldots,\overline{\beta_g}$ be the minimal
sequence of generators of $\Gamma(f)$ defined by the conditions
\begin{itemize}
\item $\overline{\beta_0}=\min (\Gamma(f)\backslash\{0\})=\ord f$,
\item $\overline{\beta_k}=\min (\Gamma(f)\backslash \bN \overline{\beta_0}+\cdots+\bN \overline{\beta_{k-1}})$ for $k\in\{1,\ldots,g\},$
\item $\Gamma(f)=\bN \overline{\beta_0}+\cdots+\bN \overline{\beta_{g}}$.
\end{itemize}

\noindent Let $e_k=\gcd(\overline{\beta_0},\ldots,\overline{\beta_k})$ for $k\in\{1,\ldots,g\}.$ Then $e_0>e_1>\cdots e_{g-1}>e_g=1$ and $e_{k-1}\overline{\beta_k}<e_{k}\overline{\beta_{k+1}}$ for $k\in\{1,\ldots,g-1\}.$ Let $n_k=e_{k-1}/e_k$ for 
$k\in\{1,\ldots,g\}.$ Then $n_k>1$ for $k\in\{1,\ldots,g\}$ and $n_k\overline{\beta_k}< \overline{\beta_{k+1}}$ for $k\in\{1,\ldots,g-1\}.$

\medskip

\noindent The degree of the conductor $c(f)$ is equal to the smallest element of $\Gamma(f)$ such that $c(f)+N\in \Gamma(f)$ for all integers $N\geq 0$. It is given by the conductor formula: $c(f)=\sum_{k=1}^g
(n_k-1){\overline{\beta_k}}-\overline{\beta_0}+1$.

\medskip

\noindent  For the proof of the above equality we refer the reader to \cite{GB-P}.

\medskip

\noindent The Milnor number $\mu(f)$ is not, in general, determined by $\Gamma(f)$. The following example is borrowed from \cite{BGreuel}: take $f=x^p+y^{p-1}$ and $g=(1+x)f$, where $p>2$. Then $\Gamma(f)=\Gamma(g)$, $\mu(f)=+\infty$ and $\mu(g)=p(p-2)$. By a plane curve singularity we mean a  nonzero power series of order greater than 1. The aim of this note is

\begin{teorema}[Main result]
\label{main result}
Let $f\in \bK[[x,y]]$ be an irreducible singula\-ri\-ty and let $\overline{\beta_0},\ldots,\overline{\beta_g}$ be the minimal
system of generators of $\Gamma(f)$. Suppose that $p=\ch \bK>\ord f$. Then the following two conditions are equivalent:
\begin{enumerate}
\item $\overline{\beta_k}\not\equiv 0$ (mod $p$) for $k\in\{1,\ldots,g\},$
\item $\mu(f)=c(f)$.
\end{enumerate}
\end{teorema}

\noindent We prove Theorem \ref{main result} in Section \ref{proof} of this note.

\begin{ejemplo}
Let $f(x,y)=(y^2+x^3)^2+x^5y$. Then $f$ is irreducible and $\Gamma(f)=4\bN+6\bN+13\bN$ (see \cite[Theorem 6.6]{GB-P}). By the conductor formula $c(f)=16$. Let $p=\ch \bK >\ord f=4$. If $p\neq 13$ then $\mu(f)=c(f)$ by Theorem \ref{main result}. If $p=13$ then a direct calculation shows that $\mu(f)=17$.
\end{ejemplo}

\begin{ejemplo}
\label{ej}
Let $f=x^m+y^n+\sum_{n\alpha+m\beta>nm} c_{\alpha\beta}x^{\alpha}y^{\beta}$, where $1<n<m$ and $\gcd(n,m)=1$. Then
$\Gamma(f)=\bN n +\bN m$ and $c(f)=(n-1)(m-1)$. We get $\mu(f)\geq (n-1)(m-1)$ with equality if and only if $n\not\equiv 0$ (mod $p$) and $m\not\equiv 0$ (mod $p$). To compute $\mu(f)$ one can use \cite[Theorem 3]{Furuya}.
\end{ejemplo}

\section{Factorization of the polar curve}
\label{polar}
\noindent Let $f\in \bK[[x,y]]$ be an irreducible singularity and let $\Gamma(f)=\bN \overline{\beta_0}+\cdots+\bN \overline{\beta_{g}}$ be the semigroup associated with $f$. Since $f$ is unitangent $i_0(f,x)=\ord f$ or $i_0(f,y)=\ord f$. In all this section we assume that $i_0(f,x)=\ord f$. Let
$n=\ord f$.

\begin{lema}
\label{igual caracteristica} Let $\psi=\psi(x,y)\in \bK[[x,y]]$ be an
irreducible power series such that $i_0(\psi,x)=\ord \psi$.  If
$\frac{i_0(f,\psi)}{\ord \psi}>\frac{e_{k-2}\overline{\beta_{k-1}}}{n}$ for $k\geq 2$ then
$\ord \psi \equiv 0$  $\left(mod \;\frac{n}{e_{k-1}}\right)$. 
\end{lema}
\begin{proof} See \cite[Lemma 5.6] {GB-P}.
\end{proof}

\medskip

\noindent In what follows we need a sharpened version of  Merle's factorization theorem (see \cite[Theorem 3.1]{Merle}).

\begin{teorema}
\label{decomposition}
Suppose that $\ord f\not \equiv 0$ (mod $p$). Then $\frac{\partial f}{\partial y}=\psi_1\cdots \psi_g$ in $\bK[[x,y]],$ where

\begin{enumerate}
\item[(i)] $\ord \psi_k=\frac{n}{e_k}-\frac{n}{e_{k-1}}$ for $k\in \{1,\ldots,g\}$.
\item[(ii)] If $\phi\in \bK[[x,y]]$ is an irreducible factor of $\psi_k$, $k\in \{1,\ldots,g\}$, then \[\frac{i_0(f,\phi)}{\ord \phi}=\frac{e_{k-1}\overline{b_k}}{n},\]
\noindent and 
\item[(iii)] $\ord \phi\equiv 0$  $\left(mod \;\frac{n}{e_{k-1}}\right)$.
\end{enumerate}
\end{teorema}

\noindent \begin{proof} The proof of the existence of the factorization $\frac{\partial f}{\partial y}=\psi_1\cdots \psi_g$ with properties $(i)$ and $(ii)$ given by Merle for the generic polar in the case $\bK=\bC$  works in our situation (see also \cite{Delgado}). To check $(iii)$ observe that  $i_0\left(\frac{\partial f}{\partial y},x\right)=n-1$ and consequently $i_0(\phi,x)=\ord \phi$ for any irreducible factor $\phi$ of $\frac{\partial f}{\partial y}$. Then use Lemma \ref{igual caracteristica}.
\end{proof}

\section{Proof of the main result}
\label{proof}
\noindent We keep the notation and assumptions of Section \ref{polar}. In particular $f\in \bK[[x,y]]$ is irreducible and $i_0(f,x)=\ord f$. We let $n=\ord f$. The following lemma is well-known and may be  deduced from the formula
$\overline{{\cal O}}f'_y={\cal C}{\cal D}_x$, where ${\cal D}_x$ is the different of $\overline{{\cal O}}$ with respect to the ring $\bK[[x]]$(see \cite[p. 10]{Zariski-1986} and \cite[Aphorism 5]{Abhyankar}).

\begin{lema} 
\label{LL}
Suppose that $\ord f\not\equiv 0$ (mod $p$). Then 
\[i_0\left(f, \frac{\partial f}{\partial y}\right)=c(f)+\ord f-1.\]
\end{lema}
\noindent \begin{proof} Since $n\not\equiv 0$ (mod $p$) the irreducible curve $f=0$ has a good parametrization of the form  $(t^n, y(t))$. Let $\beta_0=n,\beta_1,\ldots, \beta_g$ be the characteristic of $(t^n,y(t))$. Then $\overline{\beta_0}=\beta_0$,
$\overline{\beta_1}=\beta_1$ and $\overline{\beta_{k+1}}=n_k\overline{\beta_k}+\beta_{k+1}-\beta_{k}$ for $k\in \{1,\ldots, g-1\}$ (see \cite[Section 3]{Zariski-1986}). 

\medskip

\noindent Denote by $\mathbf U(n)$ the group of $n$th roots of unity in $\bK$. A simple computation shows that 

\[i_0\left(f, \frac{\partial f}{\partial y}\right)=\sum_{\epsilon \in \mathbf U(n)\backslash \{1\}}\ord (y(t)-y(\epsilon t))=\sum_{k=1}^g(e_{k-1}-e_k)\beta_k=\sum_{k=1}^g(n_k-1)\overline{\beta_k}.
\]

\noindent Now, the lemma follows from the conductor formula.
\end{proof}

\begin{coro}
If $\ord f\not\equiv 0$ (mod $p$) then $\mu(f)=c(f)$ if and only if $i_0\left(f, \frac{\partial f}{\partial y}\right)=\mu(f)+\ord f-1$.
\end{coro}

\noindent If $\ch \bK=0$ then $i_0\left(f, \frac{\partial f}{\partial y}\right)=\mu(f)+i_0(f,x)-1$ (see \cite[Chap. II, Proposition 1.2]{Teissier}) for any reduced series $f\in \bK[[x,y]]$, whence $\mu(f)=c(f)$ for irreducible $f$ in characteristic zero.

\begin{lema} 
\label{mu}
Suppose that $p>\ord f$. Then $i_0\left(f, \frac{\partial f}{\partial y}\right)\leq \mu(f)+\ord f-1$ with equality if and only if $\overline{\beta_k}\not\equiv 0$ (mod $p$) for $k\in \{1, \ldots, g\}$.
\end{lema}
\noindent \begin{proof} Let us begin with 

\medskip

\noindent{\em Claim 1:} Suppose that $p> \ord f$. Then for every irreducible factor $\phi$ of $\frac{\partial f}{\partial y}$ we have  $i_0\left(\frac{\partial f}{\partial x},\phi\right)+ \ord \phi \geq i_0(f,\phi)$ with equality if and only if $i_0(f,\phi)\not\equiv 0$ (mod $p$).

\medskip

\noindent {\em Proof of  Claim 1} Let $\phi$ be an irreducible factor of $\frac{\partial f}{\partial y}$. Then $\ord \phi \leq \ord
\frac{\partial f}{\partial y}=\ord f-1$. Let $(x(t),y(t))$ be a good parametrization of $\phi=0$. Then $\ord x(t)=i_0(x,\phi)=\ord \phi<\ord f\leq p$ and consequently $\ord x(t)\not\equiv 0$ (mod $p$) which implies $\ord x'(t)=\ord x(t)-1$. We have
\[\frac{d}{dt}f(x(t)y(t))=\frac{\partial f}{\partial x}(x(t),y(t))x'(t).\]

\noindent Taking orders gives $\ord \frac{d}{dt}f(x(t)y(t))\geq \ord f(x(t),y(t))-1,$ with equality if and only if $\ord f(x(t),y(t))\not\equiv 0$ (mod $p$), and $\ord \frac{\partial f}{\partial x}(x(t),y(t))x'(t)=\ord \frac{\partial f}{\partial x}(x(t),y(t))+\ord x(t)-1.$

\medskip

\noindent Therefore $\ord \frac{\partial f}{\partial x}(x(t),y(t))+\ord x(t)\geq \ord f(x(t),y(t))$ with equality if and only if 
$\ord f(x(t),y(t))\not\equiv 0$ (mod $p$). Passing to the intersection numbers we get the claim.

\medskip 
\noindent{\em Claim 2:} Suppose that $p> \ord f$ and let $\frac{\partial f}{\partial y}=\psi_1\cdots \psi_g$ be the Merle factorization of the polar $\frac{\partial f}{\partial y}$. Let $\phi$ be an irreducible factor of $\psi_k$. Then $i_0(f,\phi)\not\equiv 0$ (mod $p$) if and only if $i_0(f,\phi)\not\equiv 0$ (mod $\overline{\beta_k}$).

\medskip

\noindent {\em Proof of  Claim 2} By Theorem \ref{decomposition} {\em (iii)} we can write $\ord \phi=m_k\frac{n}{e_{k-1}},$ where $m_k\geq 1$ is an integer. Since $\ord \phi \leq \ord \frac{\partial f}{\partial y}=\ord f-1<p$ we have $\ord \phi \not\equiv 0$ (mod $p$) which implies $m_k\not\equiv 0$ (mod $p$). By Theorem \ref{decomposition} {\em (ii)} $i_0(f,\phi)=
\left(\frac{e_{k-1}\overline{\beta_k}}{n}\right)\ord \phi=m_k \overline{\beta_k}$. Therefore $i_0(f,\phi)\not\equiv 0$ (mod $p$) if and only if $\overline{\beta_k}\not\equiv 0$ (mod $p$).

\medskip

\noindent Now we continue with the proof of the lemma. Let $P$ be the set of all irreducible factors of $\frac{\partial f}{\partial y}$. Then, by Claim 1:

\begin{eqnarray*}
i_0\left(f,\frac{\partial f}{\partial y}\right)&=&\sum_{\phi \in P}e(\phi) i_0(f,\phi)\leq \sum_{\phi \in P}e(\phi) i_0\left(\frac{\partial f}{\partial x},\phi \right)=\mu(f)+\ord \frac{\partial f}{\partial y} \\
&=&\mu(f)+\ord f-1,
\end{eqnarray*}

\noindent where $e(\phi)=\max\left \{e\,:\,\phi^e \;\hbox{\rm divides } \frac{\partial f}{\partial y}\right \}$ and with equality if and only if $i_0(f,\phi)\not\equiv 0$ (mod $p$) for all $\phi \in P$. According to Claim 2 $i_0(f,\phi)\not\equiv 0$ (mod $p$) for all $\phi \in P$ if and only if $\overline{\beta_k}\not\equiv 0$ (mod $p$) for $k\in \{1,\ldots,g\}$ and the lemma follows.
\end{proof}

\begin{nota}
If $p<\ord f$ then the proof of Lemma \ref{mu} fails, even if $\ord f\not\equiv 0$ (mod $p$). Take $f=x^{p+2}+y^{p+1}+x^{p+1}y$.
\end{nota}

\noindent {\em Proof of Theorem \ref{MR}}
 Let $f\in \bK[[x,y]]$ be an irreducible singularity. Suppose that $p=\ch \bK> \ord f$. Then by Lemma \ref{LL} $\mu(f)=c(f)$ is equivalent to Teissier's formula $i_0\left(f, \frac{\partial f}{\partial y}\right)=\mu(f)+\ord f-1$, which by Lemma \ref{mu} holds if and only if $\overline{\beta_k}\not\equiv 0$ (mod $p$) for $k\in \{1,\ldots,g\}$.
 
 \medskip
 
\noindent {\bf Conjecture}\\  Let $f\in \bK[[x,y]]$ be an irreducible singularity with the semigroup $\Gamma(f)=\bN \overline{\beta_0}+\cdots+\bN \overline{\beta_{g}}$. Then $\mu(f)=c(f)$ if and only if $\overline{\beta_k}\not\equiv 0$ (mod 
$\ch \bK$) for $k\in \{0,\ldots,g\}$.

\medskip

\noindent The conjecture is true if $\Gamma(f)=\bN \overline{\beta_0}+\bN \overline{\beta_{1}}$ (cf. Example \ref{ej} of this note).

\medskip
\noindent
{\small Evelia Rosa Garc\'{\i}a Barroso\\
Departamento de Matem\'aticas, Estad\'{\i}stica e I.O. \\
Secci\'on de Matem\'aticas, Universidad de La Laguna\\
Apartado de Correos 456\\
38200 La Laguna, Tenerife, Espa\~na\\
e-mail: ergarcia@ull.es}

\medskip

\noindent {\small Arkadiusz P\l oski\\
Department of Mathematics and Physics\\
Kielce University of Technology\\
Al. 1000 L PP7\\
25-314 Kielce, Poland\\
e-mail: matap@tu.kielce.pl}

\end{document}